\documentclass[11pt]{amsart}
\usepackage{amssymb, epsfig}
\newcommand{\R}{\mathbf{R}}

\newcommand{\p}{\mathbf{P}}
\newcommand{\subi}{_i}
\newcommand{\tensor}{\otimes}

\theoremstyle{plain}
\newtheorem{thm}{Theorem}[section]
\newtheorem{cor}[thm]{Corollary}
\newtheorem{prop}[thm]{Proposition}
\newtheorem{lem}[thm]{Lemma}

\theoremstyle{definition}

\newtheorem{rem}[thm]{Remark}
\newtheorem{defn}[thm]{Definition}
\newtheorem{examp}[thm]{Example}

\renewcommand{\tilde}{\widetilde}

\DeclareMathOperator{\md}{md}
\DeclareMathOperator{\grad}{grad}

\DeclareMathOperator{\Ric}{Ric}

\numberwithin{equation}{section}
\newcounter{step}

\begin{document}

\title{Lorentz and semi-Riemannian spaces with Alexandrov curvature bounds}
\author{Stephanie B. Alexander}
\address{1409 W. Green St., Urbana, Illinois 61801}
\email{sba@math.uiuc.edu}
\author{Richard L. Bishop}
\address{1409 W. Green St., Urbana, Illinois 61801}
\email{bishop@math.uiuc.edu}
\thanks{The research described in this paper was made possible in part
by Award No.  RMI-2381-ST-02
of the U.S. Civilian Research \& Development
Foundation for the Independent States of the Former Soviet Union
(CRDF).
Also, important parts were developed while the first author visited
IHES in Bures-sur-Yvette, France.}
\subjclass{53B30,53C21,53B70}

\begin{abstract}
A semi-Riemannian manifold  is said to satisfy $R\ge K$
(or $R\le K$) if spacelike sectional curvatures are $\ge K$ and
timelike ones are $\le K$ (or the reverse). Such spaces are abundant,
as warped product constructions show; they include, in particular, big
bang Robertson-Walker spaces. By stability, there are many non-warped
product examples. We prove the equivalence of this type of curvature
bound with local triangle comparisons on the signed lengths of
geodesics.  Specifically, $R\ge K$ if and only if locally the signed
length of the geodesic between two points on any geodesic triangle is
at least that for the corresponding points of its model triangle in
the Riemannian, Lorentz or anti-Riemannian plane of curvature $K$ (and
the reverse for $R\le K$). The proof is by comparison of solutions of
matrix Riccati equations for a modified shape operator that is
smoothly defined along reparametrized geodesics (including null
geodesics) radiating from a point.  Also proved are semi-Riemannian
analogues to the three basic Alexandrov triangle lemmas, namely, the
realizability, hinge and straightening lemmas.  These analogues are
intuitively surprising, both in one of the quantities considered, and
also in the fact that monotonicity statements persist even though the
model space may change. Finally, the algebraic meaning of these
curvature bounds is elucidated, for example by relating them to a
curvature function on null sections.
 \end{abstract}

\maketitle

\section{Introduction}
\markboth{\sc Stephanie B. Alexander and Richard L. Bishop}
{\sc Semi-Riemannian spaces with curvature bounds}
\label{sec:intro} 

\subsection{Main theorem}\label{ss:maintheorem}

Alexandrov spaces are geodesic metric spaces with curvature bounds in
the sense of local triangle comparisons.  Specifically, let $S_K$
denote the simply connected 2-dimensional Riemannian space form of
constant curvature $K$.  For curvature bounded below (CBB) by $K$, the
distance between any two points of a geodesic triangle is required to
be $\ge$ the distance between the corresponding points on the
``model'' triangle with the same sidelengths in $S_K$.  For curvature
bounded above (CBA), substitute ``$\le$''.  Examples of Alexandrov
spaces include Riemannian manifolds with sectional curvature $\ge K$
or $\le K$. A crucial property of Alexandrov spaces is their
preservation by Gromov-Hausdorff convergence (assuming uniform
injectivity radius bounds in the CBA case).  Moreover, CBB spaces are
topologically stable in the limit \cite{P}, a fact at the root of
landmark Riemannian finiteness and recognition theorems.  (See Grove's
essay \cite{Ge}.)  CBA spaces are also important in geometric group
theory (see \cite{Gv, BH}) and harmonic map theory (see, for example,
\cite{ GvS, J, EF}).

In Lorentzian geometry, \emph{timelike} comparison and rigidity theory
is well developed.  Early advances in timelike comparison geometry
were made by Flaherty \cite{F}, Beem and Ehrlich \cite{BE}, and Harris
\cite{H1, H2}. In particular, a purely timelike, global triangle
comparison theorem was proved by Harris \cite{H1}.  A major advance in
rigidity theory was the Lorentzian splitting theorem, to which a
number of researchers contributed; see the survey in \cite{BEE}, and
also the subsequent warped product splitting theorem in \cite{AGH}.
The comparison theorems mentioned assume a bound on sectional
curvatures $K(P)$ of timelike 2-planes $P$.  Note that a bound over
\emph{all} nonsingular 2-planes forces the sectional curvature to be
constant \cite{Ki}, and so such bounds are uninteresting.

This project began with the realization that certain Lorentzian
warped products, which may be called Minkowski, de Sitter or anti-de
Sitter cones, possess a global triangle comparison property that is
not just
timelike, but is fully analogous to the Alexandrov one. The
comparisons we mean are on signed lengths of geodesics, where the
timelike sign is taken to be negative. In this paper, \emph{length} of either
geodesics or vectors is always signed, and we will not
talk about the length of nongeodesic curves.  The \emph{model spaces}
are $S_K$, $M_K$ or $-S_K$, where $M_K$ is the simply connected
$2$-dimensional Lorentz space form of constant curvature $K$, and
$-S_K$ is $S_K$ with the sign of the metric switched, a space of
constant curvature $-K$.

The cones mentioned above turn out to have sectional curvature bounds
of the following type.  For any semi-Riemannian manifold, call a
tangent section  \emph{spacelike} if the metric is definite there,
and \emph{timelike} if it is nondegenerate and indefinite.  Write
\emph{$R\ge K$} if spacelike
sectional curvatures are $\ge K$ and timelike ones are $\le K$; for
$R\le K$, reverse ``timelike'' and ``spacelike''.  Equivalently,
$R\ge K$ if the curvature tensor satisfies
\begin{equation}
\label{eq:def}
R(v,w,v,w) \ge
K (\langle v,v \rangle \langle w,w \rangle  - \langle v,w \rangle ^2),
\end{equation}
and similarly  with inequalities reversed.

The meaning of this type of curvature bound is clarified  by noting
that if one has merely  a bound above on timelike sectional
curvatures, or merely a bound below on spacelike ones, then the 
restriction $R_V$ of the sectional curvature
function to any nondegenerate $3$-plane $V$ has a curvature
bound below in
our sense: $R_V\ge K(V)$ (as follows from \cite{BP};  see
\S\ref{sec:algebra} below).  Then $R\ge K$ means that $K(V)$ may be
chosen independently of $V$.

Spaces satisfying
$R\ge K$ (or $R\le K$) are abundant, as
warped product constructions show. They include, for example, the big
bang cosmological models discussed by Hawking and Ellis \cite[p.
134-138]{HE} (see \S\ref{sec:examples} below). Since there are many warped
product examples satisfying
$R\ge K$ for all $K$ in a nontrivial finite interval,
then by stability, there are many non-warped product examples.

Searching the literature for this type of curvature bound, we found it
had been studied earlier by Andersson and Howard \cite{AH}.  Their
paper contains a Riccati equation analysis and gap rigidity theorems.
For example: A geodesically complete semi-Riemannian manifold of
dimension $n\ge 3$ and index $k$, having either $R\ge 0$ or $R\le 0$
and an end with finite fundamental group on which $R\equiv 0$, is
$\R^n_k$ \cite{AH}.   Their method uses parallel hypersurfaces, and
does not concern triangle comparisons or the methods of Alexandrov
geometry. Subsequently, D\'{\i}az-Ramos, Garc\'{\i}a-R\'{\i}o, and
Hervella obtained a volume comparison theorem for ``celestial
spheres'' (exponential images of spheres in spacelike hyperplanes)
in a Lorentz manifold with $R \ge K$ or $R \le K$
\cite{DGH}.

Does this type of curvature bound always imply local triangle
comparisons, or do triangle comparisons only arise in special cones?  In
this paper we prove that curvature bounds $R\ge K$ or
$R\le K$ are actually equivalent to local triangle comparisons. The
existence of model triangles is described in the Realizability
Lemma of \S \ref{sec:modellem}.  It states that any point in $\R^3-(0,0,0)$
represents the sidelengths of a unique triangle in a model space of
curvature $0$, and the same holds for $K\ne 0$ under appropriate \emph{size
bounds for $K$}.

We say $U$ is a
\emph{normal  neighborhood}  if it is a normal coordinate
neighborhood (the diffeomeorphic exponential image of some open
domain in the tangent space) of each of its points. There is a corresponding
distinguished geodesic between any two points of $U$, and the 
following theorem refers to  these geodesics and the triangles they 
form. If in addition the triangles
satisfy size
bounds for $K$, we say $U$ is \emph{normal for $K$}. All geodesics are assumed
parametrized by $[0,1]$, and by
\emph{corresponding} points on two geodesics, we mean
points having the same affine parameter.

\begin{thm}\label{thm:comparison}
If a semi-Riemannian manifold satisfies $R \ge K \,\,(R\le K)$, and
$U$ is a normal neighborhood for $K$, then the signed length of
the geodesic
between two points on any geodesic triangle of  $U$ is at least \,(at most)\,
that for the corresponding points on the model triangle in $S_K$,
$M_K$ or $-S_K$.

Conversely, if triangle comparisons hold in some normal neighborhood
of each point of a semi-Riemannian manifold, then $R \ge K \,\,(R\le K)$.
\end{thm}

In this paper, we restrict our attention to local triangle
comparisons (i.e.,  to normal neighborhoods) in
smooth spaces. In the Riemannian/Alexandrov theory, local triangle comparisons
have features of potential interest to semi-Riemannian and Lorentz
geometers:  they incorporate singularities, imply global comparison
theorems, and are consistent with a theory of limit spaces.  Our
longer-term goal is to  see what the extension of the  theory
presented here can contribute to similar questions in
semi-Riemannian and Lorentz geometry.


\subsection{Approach}
\label{ss:approach } 

We begin by mentioning some intuitive barriers to approaching Theorem
\ref{thm:comparison}.  In resolving them, we are
going to draw on papers by Karcher \cite{Kr} and
Andersson and Howard
\cite{AH}, putting them to different uses than were originally envisioned.

First, a
fundamental object in Riemannian theory is the locally isometrically embedded
interval, that is, the unitspeed geodesic.   These are the paths studied
in \cite{Kr} and \cite{AH}.  However, in the
semi-Riemannian case this choice constrains consideration to fields
of geodesics all having the same causal character. By contrast, our
construction, which uses affine parameters on $[0,1]$, applies
uniformly to all the geodesics radiating from a point (or orthogonally from
a nondegenerate submanifold).

Secondly, a common paradigm in Riemannian
and Alexandrov comparison theory is the construction of a curve that
is shorter than some original one, so
that the minimizing geodesic between the endpoints is even shorter.
In the Lorentz setting, this argument still works for \emph{timelike}
curves, under a causality assumption.  However, spacelike geodesics
are unstable critical points of the length functional, and so this
argument is forbidden.

Thirdly, while the comparisons we seek can be reduced in the
Riemannian setting to $1$-dimensional Riccati equations (as in
\cite{Kr}), the semi-Riemannian case seem to require matrix Riccati
equations (as in \cite{AH}).  Such increased complexity is to be
expected, since semi-Riemannian curvature bounds
below (say) have some of the qualities of
Riemannian  curvature bounds both below
\emph{and} above.

Let us start by outlining Karcher's approach to Riemannian curvature
bounds. It included a new proof of local triangle comparisons, one
that integrated infinitesimal Rauch comparisons to get distance
comparisons without using the ``forbidden argument'' mentioned above.
Such an approach, motivated by simplicity rather than necessity in the
Riemannian case, is what the semi-Riemannian case requires.

In this approach,
Alexandrov curvature bounds are  characterized by a differential
inequality. Namely, $M$ has CBB by $K$ in the triangle comparison
sense if and only if for every $q\in M$ and unit-speed geodesic
$\gamma$, the differential inequality \begin{equation}\label{eq:CB}
(f \circ\gamma)''+ Kf \circ\gamma \le 1
\end{equation}
is satisfied
           (in the barrier sense) by the following function $f=\md_Kd_q$:
\begin{equation}\label{eq:modD}
\md_Kd_q=
\begin{cases}
(1/K)(1-\cosh\sqrt{-K}d_q),\quad & K<0 \\
(1/K)(1-\cos\sqrt{K}d_q),\quad &  K>0 \\
d_q^2/2,\quad & K=0.
\end{cases}
\end{equation}

The reason for this equivalence is that the inequalities
(\ref{eq:CB}) reduce to equations in the model
spaces $S_K$; since solutions of the differential inequalities may
be compared to those of the equations, distances in $M$ may be
compared to those in $S_K$.  The functions $\md_Kd_q $ then provide a
convenient connection between triangle comparisons and curvature
bounds, since they lead via their Hessians to a
Riccati equation along radial geodesics from $q$.

We wish to view this program as a special case of a procedure on
semi-Riemannian manifolds. For a geodesic $\gamma$ parametrized
by $[0,1]$, let
\begin{equation}\label{eq:Edef1}
E(\gamma)= \langle \gamma'(0),\gamma'(0)\rangle.
\end{equation}
Thus $E(\gamma)=\pm |\gamma|^2.$
In this paper, we work with normal neighborhoods, and set
$E(p,q)=E(\gamma_{pq})$ where $\gamma_{pq}$ is the
geodesic from $p$ to $q$ that is distinguished by the normal
neighborhood.

(In a broader setting, one may instead use the definition
\begin{equation}\label{eq:Edef2}
E(p,q)=E_q(p)=\inf\{E(\gamma):\,\ \gamma \text{ is a geodesic
joining $p$ \ and
$q$}\},
\end{equation}
under hypotheses that ensure the two definitions agree locally.
In (\ref{eq:Edef2}), $E(p,q)=\infty$ if $p$ and $q$ are not
connected by a geodesic.)

Now define the
{\em modified distance function $h_{K,q}$ at $q$} by
\begin{equation}\label{eq:modE}
h_{K,q} =
\begin{cases}
(1 - \cos\sqrt{KE_q})/K =
\sum_{n=1}^\infty \frac{(-K)^{n-1}(E_q)^n}{(2n)!}, & K\ne 0\\
E_q/2, & K=0.
\end{cases}
\end{equation}
Here, the formula remains valid when the argument of cosine is imaginary,
converting $\cos$ to $\cosh$.
In the Riemannian case, $h_{K,q} =\md_Kd_q $.
The CBB
triangle comparisons we seek will be
characterized by the differential inequality
\begin{equation}\label{eq:diffineq}
(h_{K,q} \circ\gamma)'' + KE(\gamma)h_{K,q} \circ\gamma
\le E(\gamma),
\end{equation}
on any geodesic $\gamma$ parametrized by $[0,1]$.

The self-adjoint operator $S=S_{K,q}$ associated with the Hessian of $
h_{K,q}$ may be regarded as a \emph{modified shape operator}.  It has
the following properties: in the model spaces, it is a scalar
multiple of the identity on the tangent space to
$M$ at each point;  along a nonnull geodesic
from $q$, its restriction to normal vectors is a scalar multiple of
the second fundamental form of the equidistant hypersurfaces from
$q$; it is smoothly defined on the regular set of $E_q$, hence along
null geodesics from $q$ (as the second fundamental forms are not);
and finally, it satisfies a matrix Riccati equation along every
geodesic from $q$, after reparametrization as an integral curve of $\grad
h_{K,q}$.

We shall also need semi-Riemannian analogues to the
three basic triangle lemmas on which Alexandrov geometry builds,
namely, the Realizability, Hinge and
Straightening Lemmas.  The analogues
are intuitively surprising, both in one of the
quantities considered, and also in the fact that monotonicity statements
persist even though the model space may change. The  Straightening
Lemma is an indicator that, as in the standard
Riemannian/Alexandrov case, there is a singular counterpart to
the smooth theory
developed in this paper.


\subsection{Outline of paper}
\label{ss:outline} 

We begin in \S\ref{sec:modellem}  with the triangle lemmas just mentioned.
In \S\ref{sec:modE}, it is shown that the
differential inequalities  (\ref{eq:diffineq})
become equations
in the model spaces, and hence characterize our triangle comparisons.

Comparisons for the modified shape operators under
semi-Riemannian curvature bounds are  proved in
\S\ref{sec:comparericcati}, and Theorem
\ref{thm:comparison} is proved in
\S\ref{sec:comparison}.

In \S\ref{sec:algebra}, semi-Riemannian  curvature bounds are related to the
analysis by Beem and Parker of the pointwise ranges of sectional
curvature \cite{BP}, and to the ``null'' curvature bounds considered
by Uhlenbeck \cite{U} and Harris \cite{H1}.

Finally, \S\ref{sec:examples} considers examples of semi-Riemannian
spaces with curvature bounds, including Robertson-Walker ``big bang''
spacetimes.


\section{Triangle lemmas in model spaces}
\label{sec:modellem}

Say three  numbers \emph{satisfy the strict triangle inequality}
if they are positive and the largest is less than the sum of the other
two.  Denote the points of $\R^3$ whose coordinates satisfy the strict
triangle inequality by $T ^+$, and their negatives by $T^-$. A triple,
one of whose entries is the sum of the other two, will be called
\emph{degenerate}. Denote the points of $\R^3-(0,0,0)$ whose coordinates
are nonnegative degenerate triples by $D^+$, and their negatives by
$D^-$.

In Figure 1, the shaded cone is $D^+$, and the interior of its
convex hull is $T^+$.

Say a point is \emph{realized} in a model
space if its coordinates
are the sidelengths of a triangle.  As usual, set
$\pi/\sqrt{k}=\infty$ if $k\le 0$.

\begin{lem}[Realizability Lemma] \label{lem:realiz} 
Points of $\R^3-(0,0,0)$ have unique realizations, up to isometry of
the model space, as follows:
\begin{list}
{\emph{\arabic{step}.}}
{\usecounter{step}
\setlength{\rightmargin}{7mm}\setlength{\leftmargin}{7mm}}
\item
A point in $T^+$ is realized by a unique triangle in $S_K$, provided
the sum of its coordinates is $<2\pi/\sqrt{K}$.  A point in $T^-$ is
realized by a unique triangle in $-S_K$, provided the sum of its
coordinates  is $>-2\pi/\sqrt{K}$.
\item
A point in $D^+$ is realized by unique triangles in $S_K$ and $M_K$,
provided the largest coordinate is $<\pi/\sqrt{K}$. A point in $D^-$ is
realized by unique triangles in $-S_K$ and $M_K$, provided the
smallest coordinate is $>-\pi/\sqrt{K}$.
\item
A point in the complement of $T^+\cup T^-\cup D^+ \cup D^-\cup(0,0,0)$
is realized by a unique triangle in $M_0=\R^2_1$.  For $K>0$, if the
largest coordinate is $<\pi/\sqrt{K}$, the point is realized by a
unique triangle
in $M_K$.  For $K<0$, if the smallest coordinate is $> -\pi /\sqrt{-K}$,
the point is realized by a unique triangle in $M_K$.

\end{list}

\end{lem}

\begin{proof}
Part 1 is standard, as is Part 2 for $\pm S_K$.  Now consider a point
not in $T^+\cup T^-\cup (0,0,0)$, and denote its coordinates by $a\ge b\ge
c$.

To realize this point in $M_0=\R^2_1$, suppose $a>0$ and take a
segment $\gamma$ of length $a$ on the $x^1$-axis.  Since distance
``circles'' about a point $p$ are pairs of lines of slope $\pm 1$
through $p$ if the radius is $0$, and hyperbolas asymptotic to these
lines otherwise, it is easy to see that circles about the endpoints of
$\gamma$ intersect, either in two points or tangentially, subject only
to the condition that $a\ge b+c$ if $c\ge 0$, namely, the point is not
in $T^+$.  Thus our point may be realized in $R^2_1$, uniquely up to
an isometry of $R^2_1$. On the other hand, if $a\le 0$ then $c<0$, so
by switching the sign of the metric, we have just shown there is a
realization in $-\R^2_1 = \R^2_1$.

For $K>0$, $M_K$ is the simply connected cover of the quadric surface
$<p,p>=1/K$ in Minkowski $3$-space with signature $(++-)$.  Suppose
$0<a<\pi/\sqrt{K}$, and take a segment $\gamma$ of length $a$ on the
quadric's equatorial circle of length $2\pi/\sqrt{K}$ in the
$x^1x^2$-plane.  A distance circle about an endpoint of $\gamma$ is a
hyperbola or pair of lines obtained by intersection with a $2$-plane
parallel to or coinciding with the tangent plane. Two circles about
the endpoints of $\gamma$ intersect, either in two points or
tangentially, if the vertical line of intersection of their $2$-planes
cuts the quadric.  This occurs subject only to the condition that
$a\ge b+c$ if $c\ge 0$, namely, the point is not in $T^+$.  On the
other hand, if $a\le 0$ then $c<0$.  Take a segment $\gamma$ of length
$c$ in the quadric, where $\gamma$ is symmetric about the
$x^1x^2$-plane.  Circles of nonpositive radius about the endpoints of
$\gamma$ intersect if the horizontal line of intersection of their
$2$-planes cuts the quadric, and this occurs subject only to the
condition that $c<a+b$, namely, the point is not in $T^-$.

Since $M_{-K}=-M_K$, switching the sign of the metric completes the proof.
\end{proof}

Let us say the points of $\R^3-(0,0,0)$ for which Lemma
\ref{lem:realiz} gives model space realizations \emph{satisfy size
bounds for $K$} (for $K=0$, no size bounds apply). Such a point may be
expressed as $(|pq|,|qr|,|rp|)$, where $\triangle pqr$ is a realizing
triangle in a model space of curvature $K$, the geodesic
$\gamma_{pq}$ is a side parametrized by $[0,1]$ with
$\gamma_{pq}(0)=p$, and we write $|pq|=|\gamma_{pq}|$. By the
\emph{nonnormalized angle $\angle pqr$}, we mean the
inner product
$<\gamma_{qp}'(0),\gamma_{qr}'(0)>$.

In our terminology,
$\angle pqr$ is the \emph{included}, and $\angle qpr$ and $\angle qrp$
are the \emph{shoulder}, nonnormalized angles for $(|pq|,|qr|,|rp|)$.
This terminology is welldefined since the realizing model space and
triangle are uniquely determined except for degenerate triples.  The
latter have only two realizations, which lie in geodesic segments in
different model spaces but are isometric to each other.

An important ingredient of the Alexandrov theory is the Hinge Lemma
for angles in $S_K$, a monotonicity statement that follows directly
from the law of cosines. Part 1 of the following lemma is its
semi-Riemannian version.  A new ingredient of our
arguments is the use of nonnormalized shoulder angles, in which both
the ``angle'' and one side vary simultaneously.  Not only do we obtain
a monotonicity statement that for
$K\ne 0$ is not directly apparent from the law of cosines (Part 2 of
the following lemma), but we
find that monotonicity persists even as the model space changes.

\begin{lem}[Hinge Lemma]\label{lem:hinge} 
Suppose a point of $\R^3-(0,0,0)$ satisfies size bounds for $K$, and
the third coordinate varies with the first two fixed.  Denote the point by
$(|pq|,|qr|,|rp|)$ where $\triangle pqr$ lies in a possibly varying
model space of curvature $K$.
\begin{list}
{\emph{\arabic{step}.}}
{\usecounter{step}
\setlength{\rightmargin}{7mm}\setlength{\leftmargin}
{7mm}}
\item
The included nonnormalized angle $\angle pqr$ is a decreasing
function of $|pr|$.
\item
Each shoulder nonnormalized angle, $\angle qpr$ or $\angle qrp$, is
an increasing function of $|pr|$.
\end{list}
\end{lem}

\begin{proof} Suppose $K=0$.  Then the model spaces are
semi-Euclidean planes, and the sides of a triangle may be represented
by vectors $A_1$, $A_2$ and $A_1-A_2$.  Set
$a\subi=<A\subi,A\subi>$ and $c=<A_1-A_2,A_1-A_2>$, so
\begin{equation}\label{eq:ip1}
c=a_1+a_2-2<A_1,A_2 >.
\end{equation}
Since $c$ is an increasing function of its sidelength, Part $1$ in
any fixed model space is immediate by taking $a_1$
and $a_2$ in (\ref{eq:ip1}) to be fixed. For  Part $2$ in any fixed
model space, it is only necessary to rewrite (\ref{eq:ip1}) as
\begin{equation}\label{eq:ip2}
c-a_1+2<A_1,A_2> = a_2,
\end{equation}
where $a_1$ and $c$ are fixed.

A change of model space occurs when the varying point in
$\R^3-(0,0,0)$ moves upward on a vertical line $L$, and passes either
into or out of $T^+$ by crossing $D^+$ (the same argument will hold
for $T^-$ and $D^-$).  See Figure \ref{fig:hinge}.
\begin{figure}[h]
\begin{center}
\includegraphics[width= 2in]{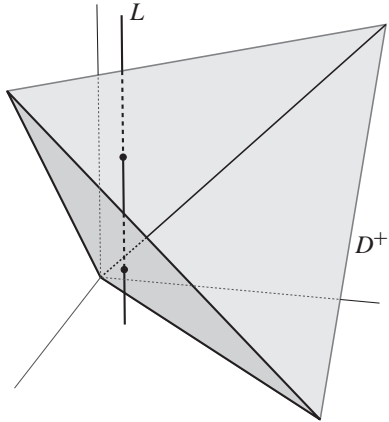}
\end{center}
\caption{Model space transitions in sidelength space}
\label{fig:hinge}
\end{figure}
             Thus $L$ is the union of three closed segments, 
intersecting only at
their two endpoints on $D^+$.  We have just seen that the included
angle function is decreasing on each segment, since the realizing
triangles are in the same model space (by choice at the endpoints and
by necessity elsewhere).  Since the values at the endpoints are the
same from left or right, the included angle function is decreasing on
all of $L$. Similarly, each shoulder angle function is increasing.

Suppose $K>0$.  The vertices of a triangle in the quadric model space
are also the vertices of a triangle in an ambient $2$-plane, whose
sides are the chords of the original sides. The length of the chord is
an increasing function of the original sidelength. Thus to derive the
lemma for $K>0$ from (\ref{eq:ip1}) and (\ref{eq:ip2}), we must verify
the following: If a triangle in a quadric model space varies with
fixed sidelengths adjacent to one vertex, and $v_1,\,v_2$ are the
tangent vectors to the sides at that vertex, then $<v_1,v_2>$ is an
increasing function of $<A_1,A_2>$ where the $A\subi$ are the chordal
vectors of the two sides.  Indeed, all points of a distance circle of
nonzero radius in the quadric model space lie at a fixed nonzero
ambient distance from the tangent plane at the centerpoint.  Thus
$A\subi$ is a linear combination of $v\subi$ and a fixed normal vector
$N$ to the tangent plane, where the coefficients depend only on the
sidelength $\ell\subi$.  The desired correlation follows.

By  switching the sign of the metric, we obtain the claim for $K<0$.
\end{proof}

\begin{rem} The Law of Cosines in a semi-Riemannian model space with
$K=0$ is (\ref{eq:ip1}).   If $K\ne 0$, the Law of Cosines for
$\triangle pqr$ may be written in unified form as follows:
\begin{align}\label{eq:coslaw}
\cos\sqrt{K E(\gamma_{pr})} = \cos\sqrt{ KE(  \gamma_{pq} ) }
&\cos\sqrt{ KE(\gamma_{qr})}\\
&-K\angle pqr
\frac{ \sin\sqrt{KE(\gamma_{pq})}} {\sqrt{KE(\gamma_{pq})  }}
\frac{ \sin\sqrt{KE(\gamma_{qr})}} {\sqrt{KE(\gamma_{qr})  }}.\nonumber
\end{align}
Here we assume
$\triangle pqr$  satisfies the size bounds for $K$.  Then  each
sidelength is $<\pi/\sqrt{K}$ if $K>0$, and $>-\pi/\sqrt{-K}$ if
$K<0$.
Part 1 of Lemma \ref{lem:hinge} can be derived
from (\ref{eq:coslaw}) as follows.  Fix
$ E(\gamma_{pq})$ and $ E(\gamma_{qr})$, and observe that
$\cos\sqrt{Kc}$ is decreasing in $c$ if $K>0$, regardless of the sign
of $c$ and even as $c$ passes through $0$, and
increasing in $c$ if $K<0$.  The size bounds imply
that the
factors $\frac{ \sin\sqrt{Ka}} {\sqrt{Ka}}$  become either
$\frac{ \sin\sqrt{Ka}} {\sqrt{Ka}}$ for $\sqrt{Ka}\in [0,\pi)$,
or $\frac{\sinh\sqrt{|Ka|}}{\sqrt{|Ka|}}$, depending on the signs of
$K$ and $a$, and hence are nonnegative.
\end{rem}

Now we are ready to prove a semi-Riemannian version of Alexandrov's
Straightening Lemma, according to which a triangle inherits
comparison properties from two smaller triangles that subdivide it.  It
turns out that the comparisons we need are on nonnormalized shoulder
angles. Moreover, the original and ``subdividing'' triangles may lie
in varying model spaces, so that geometrically we have come a long
way  from the original interpretation in terms of hinged rods.

Since geodesics are parametrized by $[0,1]$, a
point $m$ on a directed side of a triangle inherits an affine
parameter $\lambda_m\in [0,1]$.

\begin{lem}[Straightening Lemma for Shoulder Angles] \label{lem:snap}
Suppose $\triangle\widetilde{p}\widetilde{q}
\widetilde{r}$ is a triangle satisfying size bounds for $K$ in a model
space of curvature $K$.  Let $\widetilde{m}$ be a point on side
$\widetilde{p}\,\widetilde{r}$, and set $\lambda =
\lambda_{\widetilde{m}}$. Let $\triangle q_1p_1m_1$ and $\triangle
q_2m_2r_2$ be triangles in respective model spaces of curvature $K$,
where $|q_1m_1|=|q_2m_2|=|\widetilde{q} \widetilde{m}|$,
$|q_1p_1|=|\widetilde{q}\,\widetilde{p}|$,
$|q_2r_2|=|\widetilde{q}\,\widetilde{r}|$,
$|p_1m_1|=|\widetilde{p}\widetilde{m}|$,
and $|m_2r_2|=|\widetilde{m}\widetilde{r}|$.  Assume $|q_i
m_i|<\pi/\sqrt{K}$ if $K>0$, and $|q_i m_i|> -\pi/\sqrt{-K}$
if $K<0$. If
$$(1-\lambda)\,\angle p_1m_1q_1 + \lambda\,\angle
r_2m_2q_2 \ge 0,$$
then
$$\angle
\widetilde{q}\,\widetilde{p}\widetilde{m} \ge \angle q_1p_1m_1
\;\textit{ and }\;
\angle \widetilde{q}\,\widetilde{r}\widetilde{m} \ge \angle q_2r_2m_2.$$
The same statement holds with all inequalities reversed.
\end{lem}

\begin{proof}
By the definition of nonnormalized angles, $(1-\lambda)\,\angle
\widetilde{q}\widetilde{m}\widetilde{p} + \lambda\,\angle
\widetilde{r}\widetilde{m}\widetilde{p} =0$.  Therefore, by
hypothesis, either $\angle q_1m_1p_1 \ge \angle
\widetilde{q}\widetilde{m}\widetilde{p}$ or $\angle r_2m_2p_2 \ge
\angle \widetilde{q}\widetilde{m}\widetilde{p}$.  By Lemma
\ref{lem:hinge}.2, the inequality $|p\subi m\subi|\ge
|\widetilde{p}\widetilde{m}|$ holds for either $i=1$ or $i=2$, and
hence for both. But then by Lemma \ref{lem:hinge}.1, the claim
follows.
\end{proof}


\section{Modified distance functions on model spaces}
\label{sec:modE}

In this section we give a unified proof that in the model spaces of
curvature $K$, the
restrictions to geodesics $\gamma$ of the modified distance functions
$h_{K,q} $
defined by (\ref{eq:modE}) satisfy the differential equation
\begin{equation}\label{eq:modKaffine}
(h_{K,q} \circ\gamma)'' + K\langle \gamma',\gamma'\rangle h_{K,q}\circ\gamma
= \langle \gamma',\gamma'\rangle.
\end{equation}

We begin by constructing the $K$-affine functions on the model spaces.
For intrinsic metric spaces the notion of a $K$-affine function was considered
in \cite{AB1} and their structural implications were pursued in
\cite{AB2}. For semi-Riemannian manifolds the definition should be formulated
to account for the causal character of geodesics, as follows.

\begin{defn} A {\em $K$-affine function} on a semi-Riemannian manifold
is a real-valued function $f$ such that for every geodesic $\gamma$ the
restriction satisfies
\begin{equation}\label{eq:Kaffine}
(f\circ\gamma)'' + K\langle\gamma',\gamma'\rangle f = 0.
\end{equation}
We say $f$ is \emph{$K$-concave} if ``$\le 0$'' holds in
(\ref{eq:Kaffine}), and \emph{$K$-convex} if ``$\ge 0$'' holds.
\end{defn}

          (Elsewhere we have called the latter classes
\emph{$\mathcal{F}(K)$-concave/convex}.)

As in the Riemannian case, the $n$-dimensional model spaces of
curvature $K$ carry an
$n+1$-dimensional vector space of $K$-affine functions, namely, the
space of restrictions of linear functionals in the ambient semi-Euclidean
space of a quadric surface model.

Specifically, let $ \R^{n+1}_k$ be the semi-Euclidean space of
index $k$.  For
$K\ne 0$, set $Q_K= \{p \in \R^{n+1}_k : \langle p,p\rangle =
\frac{1}{K}\}$, with the induced semi-Riemannian metric, so that $Q_K$
is an $n$-dimensional space of constant curvature $K$.  (The
$2$-dimensional model spaces $M_K$ are the universal covers of such
quadric surfaces.)
For $q\in Q_K $, let $\ell_{K,q}:Q_K\to\R$ be the restriction to $Q_K$ of the
linear functional on $\R^{n+1}_k $
dual to the element $q$, namely, $\ell_{K,q}(p) = \langle q,p\rangle$.
Define $E_q$ on $Q_K$ by (\ref{eq:Edef2}).

\begin{prop} \label{prop:Kaffine}
For $K\ne 0$, the function $\ell_{K,q}$ on $Q_K$ is $K$-affine.  
For any $p$ that is
joined to $q$ by a geodesic in $Q_K$,
$$\ell_{K,q}(p) = \frac{1}{K}\cos\sqrt{KE_q(p)},$$
where the argument of cosine
may be imaginary.
\end{prop}

\begin{proof} We use the customary identification of elements of $\R^{n+1}_k $
with tangent vectors to $\R^{n+1}_k $and $Q_K$. Then the gradient
of the linear functional $\langle q,\cdot\rangle$ on $\R^{n+1}_k $ is $q$,
viewed as a parallel vector field. For $p \in Q_K$,
projection $\pi_p:T_p\R^{n+1}_k \to T_pQ_K$ is given by
$\pi_p(v) = v - K\langle v,p\rangle p$. In particular, $\pi_pp=0$. It
is easily checked that
$\grad_p\ell_{K,q} = \pi_p q$.

The connection $\nabla$ of $Q_K$ is related to the connection  $D $
of $\R^{n+1}_k $ by
projection, that is, $\nabla_v X = \pi_pD_v X$ for $v\in T_pQ_K$.
Writing $p = \gamma(t)$, $v = \gamma'(t)$ for a geodesic $\gamma$ of
$Q_K$, then
\begin{align}
(\ell_{K,q}\circ\gamma)''(t) = \langle \nabla_v\grad\ell_{K,q} , v \rangle
           &= \langle \pi_pD_v\pi_p q , v \rangle \nonumber \\
		&= \langle\pi_pD_v(q - K \langle q,p\rangle p), v
\rangle \nonumber \\
		&= \langle \pi_p (0 - K\langle q,v\rangle p - K\langle
q,p\rangle \nabla_vp), v \rangle \nonumber \\
		&= -K\langle v,v\rangle \ell_{K,q}(\gamma (t)).\label{eq:hess}
\end{align}
Thus $\ell_{K,q}$ is $K$-affine.

Since $q$ is orthogonal to the tangent
plane $T_qQ_K$, the derivatives of $\ell_{K,q}$ at $q$ are all $0$.
Along a geodesic
$\gamma$ in $Q_K$ that starts at $q$, the initial conditions for
$\ell_{K,q}\circ\gamma$ are
$\ell_{K,q}(q) = 1/K$, $(\ell_{K,q}\circ\gamma)'(v) = 0$, so
the formula for $\ell_{K,q}\circ\gamma(t)$ is
$\cos(\sqrt{K\langle v, v\rangle}\, t)/K$.
\end{proof}

For the case $K = 0$ we consider the quadric surface model to be a
hyperplane not through the origin, so that the affine functions on it
are trivially the restrictions of linear functionals.

On a model space $Q_K$ of curvature $K\ne 0$, the modified distance
function $h_{K,q} $ defined by (\ref{eq:modE}) may be written on its
domain as
\begin{equation}
\label{eq:modelh}
h_{K,q} = -\ell_{K,q} + 1/K,
\end{equation}
and  satisfies
the same differential equation along geodesics as
$\ell_{K,q}$ except for an
additional constant term, that is, $h_{K,q} $ satisfies
(\ref{eq:modKaffine}).
It is trivial to check that this equation holds when $K = 0$
and  $h_{K,q} =E_q/2$.


\section{
Ricatti comparisons for modified shape operators}
\label{sec:comparericcati}

In a given semi-Riemannian manifold $M$, set
$h=h_{K,q}$ (as in (\ref{eq:modE})) for some fixed choice of
$K$ and $q$.  Define the
\emph{modified shape operator} $S=S_{K,q}$, on the region where $h$ is smooth,
to be the self-adjoint operator associated with the
Hessian of $h$, namely,
\begin{equation}
\label{eq:Sdef1}
Sv=\nabla _v\grad h.
\end{equation}
The form of  $h$ was
chosen so that in a \emph{model space} $Q_K$, $S$ is
always a scalar multiple of the identity.
Indeed, at any point in $Q_K$,
\begin{equation}
S=
\begin{cases}
I,\quad & \text{if \,} K=0, \\
K\ell_{K,q} \cdot I,\quad & \text{if \,} K\ne 0,
\label{eq:modelS}
\end{cases}
\end{equation}
where the latter equality is by Proposition \ref{prop:Kaffine} and
(\ref{eq:hess}).

Below, our  Riccati equation (\ref{eq:riccati}) along radial
geodesics $\sigma$ from $q$  differs
from the standard one in  \cite{AH} and
\cite{Kr}, being adjusted to facilitate the proof of Theorem
\ref{thm:comparison}.
Thus it applies even if $\sigma$ is null; it concerns an operator
$S$ that is defined on the whole
tangent space; when $\sigma$ is nonnull, the restriction of $S$ to the normal
space of $\sigma$ does not agree with the second fundamental
form of the equidistant hypersurface but rather
with a rescaling of it;  and we do not differentiate with
respect to an affine parameter along $\sigma$,
but rather use the integral curve parameter of
$\grad h$.

The gradient vector field $G=\grad h$ is tangent to the radial
geodesics from $q$.  Note that $G$ is nonzero along null geodesics
radiating from $q$ even though $h$ vanishes along such geodesics.
Specifically, $G$ may be expressed in terms of $\grad E_q$ on a
normal coordinate neighborhood via (\ref{eq:modE}).  Here $\grad
E_q=2P$, where $P$ is the  image under $d\exp_q$ of the position vector
field $v\mapsto v_v$ on $T_qM$ (see \cite[p. 128]{O'N}).  If $K=0$,
then $G= P$, and
an affine
parameter $t$ on a radial geodesic from $q$ is given in terms of the
integral curve parameter $u$ of $G$ by $t=ae^u$ with
$u=-\infty$ at $0$.  If $K\ne 0$, then
$G=(\sin\sqrt{KE_q}/\sqrt{KE_q})P$, so $G$ agrees with $P$ up to
higher  order terms, and the dominant
term at $q$ in the integral
curve expression is an exponential.

Let $R_G$ be the self-adjoint {\em Ricci operator},
$R_Gv = R(G,v)G$.
We are going to establish comparisons on modified shape operators,
governed by comparisons on Ricci operators.
Since we are  interested in comparisons along two given geodesics,
each radiating from a given basepoint, the effect of restricting to
normal coordinate neighborhoods in the following proposition is
merely to rule out conjugate points along both geodesics.

\begin{prop} \label{prop:riccati}
In a semi-Riemannian manifold $M$, on a normal coordinate
neighborhood of $q$, the modified shape operator $S$ satisfies the
first-order PDE
\begin{equation}\label{eq:riccati}
\nabla_GS + S^2 - (1-Kh)S + R_G + Kdh\tensor G = 0.
\end{equation}
\end{prop}

Before verifying Proposition \ref{prop:riccati}, we shift to the
general setting of systems of ordinary differential equations in
order to summarize  all we need about Jacobi and Riccati equations.

\begin{lem}
\label{lem:jacobiriccati}
For  self-adjoint linear maps $R(t)$  on a semi-Euclidean space,
suppose $F(t)$ satisfies
\begin{equation}
\label{eq:jacobi}
F''(t)+R(t)F(t)=0
\end{equation}
for $t\in [0,b]$, where $F(0)=0$,\, $F'(0)$ is invertible, and $F(t)$
is invertible for all $t\in (0,b]$.
For a given function $g:[0,b]\to\R$ with $g(0)=0$,\, $g'(0)=1$, and
$g>0$ on $(0,b]$,  define $S$ by
\begin{equation}
\label{eq:Sdef2}
g(t)F'(t)=S(t)F(t) \text{ for } t\in (0,b],
\end{equation}
and
\begin{equation}
\label{eq:Sdef2(0)}
S(0)=I.
\end{equation}
Then $S$ is self-adjoint, smooth on $[0,b]$, and satisfies
\begin{equation}
\label{eq:generalriccati}
gS'+S^2-g'S+g^2R=0.
\end{equation}
\end{lem}

\begin{proof}
Self-adjointness of $S$ follows from (\ref{eq:jacobi}) and
self-adjointness of $R$ (see \cite[p. 839]{AH}).
By (\ref{eq:Sdef2}) and (\ref{eq:jacobi}), on $(0,b]$ we have
\begin{align*}
S'F + g^{-1} S^2F &= S'F + SF' = g'F' + g F'' \\
        &= g'F' - g RF = g'g^{-1}SF-gRF.
\end{align*}
Multiplying the first and last expressions by $gF^{-1}$ on the right
yields (\ref{eq:generalriccati}).

On  $[0,b]$ we have $g=t\overline{g} $ where $\overline{g}(0)=g'(0)=1$, and
$F=t\overline{F} $ where $\overline{F} (0)=F'(0)$ is invertible.
Then (\ref{eq:Sdef2}) gives $t\overline{g}F'= St\overline{F} $ on $(0,b]$.
By
(\ref{eq:Sdef2(0)}),
$S=\overline{g}F'\overline{F}^{-1}$ on $[0,b]$, so $S$ is smooth  there.
\end{proof}

Comparisons of solutions of (\ref{eq:generalriccati}) will be in
terms of the notion of positive definite
and positive semi-definite self-adjoint operators \cite[p. 838]{AH}. A
linear operator $A$ on a semi-Euclidean space is {\em positive definite} if
$\langle Av, v \rangle > 0$ for every $v \ne 0$, {\em positive
semi-definite} if $\langle Av,v \rangle \ge 0$. We then write $A<B$ if
$B-A$ is positive definite, and similarly for $A \le B$. Note that the
identity map $I$ is not positive definite if the index is positive;
however, the
eigenvalues of a positive definite operator $A$ are real.  If
$A \ge 0$ and $\langle Av,v\rangle = 0$, then $Av = 0$.

In \cite[p. 846-847]{AH}, a comparison theorem for the shape operators
of tubes in semi-Riemannian manifolds is stated without proof. For
the proof of Theorem \ref{thm:comparison}
we require a stronger version of the special case in which the central
submanifolds are just points, so the shape operators of
distance-spheres are compared; the strengthening comes from the extension
to modified shape operators.
Since
it is a key result for us, we now show how this version can be
derived from a modification of
the comparison theorem proved in \cite[p. 838-841]{AH}, together with a
Taylor series argument to cover the behavior at the base-point
singularity.

\begin{thm}\label{thm:AHcompare}
Let $g$ and $R\subi,\, F\subi,\, S\subi$ ($i=1,2$) be as in Lemma
\ref{lem:jacobiriccati}, and assume $g''(0)=0$. If
$R_1(t) \le R_2(t)$ for all $t \in [0,b]$, then $S_1(t) \ge S_2(t)$
on $[0,b]$. If $S_1(b) = S_2(b)$, then
$R_1(t) =R_2(t)$ on $[0,b]$.
\end{thm}

\begin{proof} First we show that (\ref{eq:generalriccati}) and the
initial data for $g$ imply
\begin{equation}
\label{eq:a}
S'(0) = 0
\end{equation}
and
\begin{equation}
\label{eq:b}
S''(0) = \frac{1}{3}(g'''(0)I -2R(0)).
\end{equation}
To see this, differentiate
(\ref{eq:generalriccati}),
obtaining
\begin{equation*}
g'S' + gS'' +S' S + SS' -g''S
- g'S' + (g^2R)'=0.
\end{equation*}
Applying the initial data for $g$ and $S(0)=I$ gives (\ref{eq:a}).
Now cancel the $\pm  g'S'$ terms and differentiate again:
$$g'S'' + gS''' +2{S'}^2 + S S''
+S''S  -g'''S -g''S' + (g^2R)''=0.$$
Setting $t=0$ gives (\ref{eq:b}).

Now for $\delta > 0$, let
$R_\delta = R_2 + \delta B$, where $B$ is a positive definite
self-adjoint operator, constant as a function of $t$.
The solutions $F_\delta$ of
$F''+R_\delta F=0$ with $F_\delta(0)=0$ and $F_\delta'(0)=F_2'(0)=
\overline{F_2} (0)$ depend continuously on the parameter $\delta$,
approaching the solution $F_2$ of $F''(t)+R_2F=0$. In particular,
$F_\delta(t)$ is invertible for all $t\in [0,b]$ if $\delta$ is
sufficiently small.
Define $S_\delta(t)$ as in (\ref{eq:jacobi}),\, (\ref{eq:Sdef2}) with
$R=R_\delta$.

Since  $ R_\delta (0) > R_2(0) \ge R_1(0)$, setting $S=S_\delta$ and
$S=S_1$ in (\ref{eq:b}) implies   $S_1''(0)  > S_\delta''(0)$.
Since  $S_1(0)=I  = S_\delta(0)$, and $S_1'(0)  = 0=S_\delta'(0)$ by
(\ref{eq:a}), then $S_1(t) > S_\delta(t)$ for all $t\in (0,a)$, where
$a>0$ depends on $\delta$.

But then  $S_1(t) > S_\delta(t)$ for $t \in (0,b]$. Our argument for
this follows \cite[p. 839]{AH}, except for showing that the
additional linear term in (\ref{eq:generalriccati}) is  harmless.
Namely, assume the statement is false. Then there exists
$t_0\in (a,b]$ for which  $S_1(t_0) \ge S_\delta(t_0)$, $S_1(t_0)-
S_\delta(t_0)$ is not positive
definite, and $S_1(t) > S_\delta(t)$ for $t < t_0$. Hence there is a
nonzero vector
$x_0$ such that $\langle (S_1(t_0) - S_\delta(t_0))x_0,x_0\rangle = 0$,
and so   $S_1(t_0)x_0 = S_\delta(t_0)x_0$.
For $f(t) = \langle(S_1(t) - S_\delta(t))x_0,x_0\rangle$, then by
(\ref{eq:generalriccati}),
\begin{align*}
g(t_0)f'(t_0) &= \langle(g(t_0)S_1'(t_0) - g(t_0)S_\delta'(t_0))x_0,
x_0\rangle \\
                    &= \langle S_\delta(t_0)x_0, S_\delta(t_0)x_0\rangle
                       - \langle S_1(t_0)x_0, S_1(t_0)x_0\rangle \\
                    &+\langle g'(t_0)(S_1(t_0) -
S_\delta(t_0))x_0,x_0\rangle + g(t_0)^2\langle (R_\delta(t_0) -
R_1(t_0))x_0,x_0\rangle\\
                    &=g(t_0)^2 \langle (R_\delta(t_0) -
R_1(t_0))x_0,x_0\rangle >0.
\end{align*}
This contradicts $g(t_0)f'(t_0) \le 0$, which is true because $f(t) > 0$ on
$(a,t_0)$ and $f(t_0) = 0$.

Since
$S_1(t) > S_\delta(t)$ for all $t \in (0,b]$, and $S_\delta(t) \to S_2(t)$
for all $t \in [0,b]$, we have $S_1(t) \ge S_2(t), \, t \in [0,b]$.
\end{proof}

Returning to the geometric setting, let us verify
Proposition \ref{prop:riccati}.

\emph{Proof of Proposition \ref{prop:riccati}.}
Let $N$ be the unit radial vector field tangent to nonnull geodesics
from $q$.
By continuity, it suffices to verify (\ref{eq:riccati}) at every
point that is joined to $q$
by a nonnull geodesic $\sigma$.

First we check that (\ref{eq:riccati}) holds when applied to
$\sigma'=N$.  Note that the modified shape operator $S$ satisfies
\begin{equation}
\label{eq:radial}
SN=\nabla _NG=(1-Kh)N.
\end{equation}
Indeed, the form of $\nabla _NG$ along a unitspeed radial geodesic
from the basepoint is the same in all manifolds, hence the same in
$M$ as in a model space.  But in a model space, (\ref{eq:modelS}) and
(\ref{eq:modelh}) imply $\nabla_NG=SN=K\ell_{K,q}N = (1-Kh)N$.  Therefore
\begin{align*}
(\nabla_GS + &S^2 - (1-Kh)S + R_G +
Kdh\tensor G)N\\
&= -K(Gh)N + (1-Kh)^2N - (1-Kh)^2N +0 +K(Nh)G\\
&= -Kg(Nh)N + K(Nh)gN = 0,
\end{align*}
as required.

Now we verify that (\ref{eq:riccati}) holds on
$V=V_{\sigma(t)}=\sigma'(t)^\perp$.  If $M$ has dimension $n$ and
index $k$,
consider an isometry  $\varphi : T_qM \to \R_k^n$. For a nonnull,
unitspeed geodesic $\sigma$ in $M$
radiating from $q$, identify $T_{\sigma(t)}M$ with $\R_k^n$ by
parallel translation to
the base point composed with $\varphi$.
Thus we identify linear operators on $T_{\sigma(t)}M$ and $\R_k^n$,
and likewise on  $V_{\sigma(t)}$ and the corresponding
$(n-1)$-dimensional subspace of $\R_k^n$.  If we restrict to
$V=V_{\sigma(t)}$, and set  $R=R_{\sigma'}$ and $g=1$, then
(\ref{eq:jacobi}) becomes the Jacobi equation for normal Jacobi
fields, and the operator defined by (\ref{eq:Sdef2}) is $S(t)=W(t)$,
the Weingarten operator, for $t>0$:
$$Wv=\nabla_vN,\, v\in V.$$
(See \cite{AH}, which uses the opposite sign convention for $W$.)  If
instead we set $R=R_{\sigma'}$ as before but
$g=|<G,G>|^{\frac{1}{2}}$ where $G=\grad h$, so that $G=gN$ and
$vg=0$ for $v\in V$, then the operator $S(t)$ defined by
(\ref{eq:Sdef2}) and (\ref{eq:Sdef2(0)}) is the restriction to $V$ of
the modified shape operator, for $t\ge 0$. Indeed, (\ref{eq:Sdef2})
implies $S(t)=g(t)W(t)$ for $t>0$, hence
\begin{equation*}
Sv=g\nabla_vN=\nabla_v(gN)=\nabla_vG,
\end{equation*}
which agrees with the definition (\ref{eq:Sdef1}) of the modified
shape operator.  And the modified shape operator  is the identity at
$q$ by (\ref{eq:radial}), since $N$ can be chosen to be any unit
vector at $q$. Then it is straightforward from
(\ref{eq:generalriccati}) that the restriction to $V$ of the modified
shape operator satisfies (\ref{eq:riccati}).

The proof of the rigidity statement proceeds just as in \cite[p. 840]{AH}.
\qed

\begin{rem}
To summarize, \cite[Theorem
3.2]{AH} applies to the Weingarten operator of the equidistant
hypersurfaces from a \emph{hypersurface}. In that case, both $R$ and
$W(0)$ are perturbed in order to obtain a strict inequality on
operators; if instead we
considered the modified Weingarten operator $S=gW$, so $S(0)=0$, we
would perturb $R$ and $S'(0)$.
On the other hand, Theorem \ref{thm:AHcompare} above applies to $gW$,
where $W$ is the Weingarten operator of the  equidistant
hypersurfaces from a \emph{point}. Here we
had $S(0)=I$ and $ S'(0)=0$, and showed that  merely perturbing $R$
implied a desired perturbation of $S''(0)$ and hence of $S(a)$ for
small $a$.
The theorem stated without
proof in \cite[p.846-847]{AH} applies to the intermediate case of
equidistant hypersurfaces from any submanifold $L$.  Except for
changes in details, our proof above works for  that
case as well.
\end{rem}

Now let us compare modified shape operators via  Theorem \ref{thm:AHcompare}.
We say two geodesic segments $\sigma$ and $\tilde{\sigma}$ in
semi-Riemannian manifolds $M$ and $\tilde{M}$ \emph{correspond} if
they are defined on the same affine parameter interval and satisfy
$\langle\sigma',\sigma'\rangle
=\langle\tilde{\sigma}',\tilde{\sigma}'\rangle$.

\begin{cor}
\label{cor:Scompare1}
For semi-Riemannian manifolds $M$ and $\tilde{M}$ of the same
dimension and index, suppose $\sigma$ and $\tilde{\sigma}$ are
corresponding nonnull geodesic segments radiating from the
basepoints $q\in M$ and $\tilde{q}\in\tilde{M}$ and having no
conjugate points. Identify linear operators on $T_{\sigma(t)}M$ with
those on  $T_{\tilde{\sigma}(t)}\tilde{M}$ by parallel translation to the
basepoints, together with an isometry of $T_qM$ and
$T_{\tilde{q}}\tilde{M}$ that identifies $\sigma'(0)$ and
$\tilde{\sigma}'(0)$.  If $R_{\sigma'}\ge\tilde{R}_{\tilde{\sigma}'}$
at corresponding points of $\sigma$ and  $\tilde{\sigma}$, then the modified
shape operators satisfy $S\le\tilde{S}$ at corresponding points of
$\sigma$ and  $\tilde{\sigma}$.
\end{cor}

\begin{proof}
The modified shape operators split into direct summands,
corresponding to their action on the one-dimensional spaces tangent
to the radial geodesics and on the orthogonal complements $V$.  The
first summand is the same for both $M $ and $\tilde{M}$.  The  second
summand is as described in Lemma \ref{lem:jacobiriccati} with
$R=R_{\sigma'}$ and  $g=|<G,G>|^{\frac{1}{2}}$.  (Since our
identification of $T_{\sigma(t)}M$ and
$T_{\tilde{\sigma}(t)}\tilde{M}$ identifies $G$ and $\tilde{G}$, we
denote both of these by $G$.) Furthermore, $g'=1-Kh$ by
(\ref{eq:radial}),  so $g''(0)=0$ by (\ref{eq:modE}).   Therefore the
corollary follows from Theorem \ref{thm:AHcompare}.
\end{proof}

\begin{cor}
\label{cor:Scompare2}
Suppose $M$ is a semi-Riemannian manifold satisfying $R\ge K$, and
$\tilde{M}=Q_K$ has the same dimension and index as $M$ and constant
curvature $K$. Then for any $p\in M$ that is joined to $q$ by a
geodesic that has no conjugate points and such that a corresponding
geodesic segment in $\tilde{M}$ has no conjugate points, the
modified shape operator $S=S_{K,q}$ satisfies
\begin{equation}
\label{eq:Scompare2}
S(p)\le (1 - Kh_{K,q} (p))\cdot I.
\end{equation}
The same statement holds with inequalities reversed.
\end{cor}

\begin{proof} Let $\sigma$ be the given geodesic from $q$ to
$p=\sigma(t)$, and $\tilde{\sigma}$ be a corresponding geodesic from
$\tilde{q}\in \tilde{M}$ to $\tilde{p}=\tilde{\sigma}(t)$. If
$\sigma$ is nonnull, then by Corollary \ref{cor:Scompare1},
(\ref{eq:modelS}) and (\ref{eq:modelh}), we have
\begin{align*}
S(p) \le \tilde{S}(\tilde{p})
&= K\ell_{K,\tilde{q}}(\tilde{p})\cdot \tilde{I}\\
&=(1 - K\tilde{h}_{K,\tilde{q}} (\tilde{p}))\cdot \tilde{I },
\end{align*}
where $\tilde{I}$ denotes the identity  operator on
$T_{\tilde{p}}\tilde{M}$, and $T_pM,\, T_{\tilde{p}}\tilde{M}$ are
identified by parallel translation to $q,\,\tilde{q}$ followed by an
isometry identifying $\sigma'(0),\, \tilde{\sigma}'(0)$. Corollary
\ref{cor:Scompare1} applies here because the righthand side of
(\ref{eq:def}) is $\tilde{R}(v,w,v,w)$,
and so $R_{\sigma'}\le\tilde{R}_{\tilde{\sigma}'}$
at corresponding points of $\sigma$ and  $\tilde{\sigma}$. Since
$\tilde{h}_{K,\tilde{q}}(\tilde{p})=
h_{K,q}(p)$, then (\ref{eq:Scompare2}) holds at $p$. Therefore
(\ref{eq:Scompare2}) holds everywhere  by continuity.
\end{proof}


\section{Proof of Theorem
\ref{thm:comparison}}\label{sec:comparison}

Now we are ready to prove that in a semi-Riemannian manifold $M$,
triangle comparisons hold in any normal neighborhood
$U$ in which there is a curvature bound $K$
and triangles satisfy size bounds for $K$. By the Realizability
Lemma, such a $\triangle pqr$ has a
model triangle $\triangle \tilde{p}\, \tilde{q} \,\tilde{r}$, which
in this section we embed in $
Q_K$,
where $Q_K$ is taken of the same dimension
and index as $U$.

There are several equivalent formulations of the triangle comparisons we seek:

\begin{prop}\label{prop:equivcomp}
The following conditions on all triangles in $U$ are equivalent:
\begin{list}
{\emph{\arabic{step}.}}
{\usecounter{step}
\setlength{\rightmargin}{7mm}\setlength{\leftmargin}{7mm}}
\item The
signed distance between any two points is  $\ge$ ($\le$) the
signed distance between the corresponding points in the model
triangle.
\item The
signed distance from  any vertex to any point on the opposite side is
$\ge$ ($\le$) the
signed distance between the corresponding points in the model
triangle.
\item The nonnormalized angles are $\le$ ($\ge$) the
corresponding nonnormalized angles of the model triangle.
\end{list}
\end{prop}

\begin{proof}
$1$ obviously implies $2$.   Conversely, for $\triangle pqr$ in $U$,
suppose $m$ is on side $\gamma_{pr}$ and $n$ is on side
$\gamma_{pq}$, and  $\lambda_m$ and $\lambda_n$ are the corresponding
affine parameters.  Let $\triangle \tilde{p}\,\tilde{q}\tilde{r}$ be
the model triangle for $\triangle pqr$, $\triangle
\tilde{p}'\,\tilde{m}'\tilde{q}'$ be the model triangle for
$\triangle pmq$, and
$\triangle\overline{p}\,\overline{m}\,\overline{n}$ be the model
triangle for $\triangle pmn$.
Let $\tilde{m}$ on $\gamma_{\tilde{p}\,\tilde{r}}$ and $\tilde{n}$ on
$\gamma_{\tilde{p}\,\tilde{q}}$ have affine parameters $\lambda_m$
and $\lambda_n$, and similarly for $\tilde{n}'$ on
$\gamma_{\tilde{p}'\,\tilde{q}'}$. By 2,
$|\overline{m}\,\overline{n}|=|mn|\ge |\tilde{m}'  \tilde{n}'|$.
Therefore by Lemma \ref{lem:hinge}.1 (Hinge),
\begin{equation}\label{eq:comp1}
\angle \overline{m}\,\overline{p}\,\overline{n}
\le \angle\tilde{m}'\,\tilde{p}'\,\tilde{n}'.
\end{equation}
Again by 2, $|\tilde{m}'\,\tilde{q}'|=|mq|\ge
|\tilde{m}\,\tilde{q}|$.  By Hinge
applied to $\triangle pmq$, together with (\ref{eq:comp1}), we have
\begin{equation}
\angle \overline{m}\,\overline{p}\,\overline{n}
\le \angle\tilde{m}'\,\tilde{p}'\,\tilde{n}'
\le \angle\tilde{m}\,\tilde{p}\,\tilde{n}.
\end{equation}
Again by Hinge, $|mn|=|\overline{m}\,\overline{n}|\ge
|\tilde{m}\,\tilde{n}|$, and so 2 implies 1.

The implication $2 \Rightarrow 3$ is a direct consequence of the
first variation formula (see \cite[p. 289]{O'N}):
\begin{equation}
(E_q\circ\gamma_{pr})'(0)=2\angle qpr.
\end{equation}
(Note that our definition of $E$ and O'Neill's differ by a factor of $2$.)

Conversely, using the same triangle notation as above, 3 gives
$\angle pmq\le \angle
\tilde{p}'\,\tilde{m}'\,\tilde{q}'$, and similarly
$\angle qmr \le \angle
\tilde{q}'\,\tilde{m}'\,\tilde{r}'$.
Since $(1-\lambda_m)\,\angle
pmq + \lambda\,\angle
qmr =0$,
we have $(1-\lambda_m)\angle
\tilde{p}'\,\tilde{m}'\,\tilde{q}'
+ \lambda_m \angle
\tilde{q}'\,\tilde{m}'\,\tilde{r}'\ge 0$.  By Lemma \ref{lem:snap}
(Straightening), $\angle\widetilde{q}'\,\widetilde{p}'
\widetilde{m}' \le \angle
\widetilde{q}\,\widetilde{p}\widetilde{m} $.  Therefore by Hinge,
$|qm|=|\widetilde{q}'\,\widetilde{m}'| \ge
|\widetilde{q}\,\widetilde{m}|$, and so 3 implies 2.
\end{proof}

Turning to the proof of Theorem \ref{thm:comparison}, consider
$\triangle pqr$ in $U$, and its model triangle
$\tilde{p}\,\tilde{q}\,\tilde{r}$, which we regard as lying in
$\tilde{M}=Q_K$.  Taking $q$ and $\tilde{q}$ as base points gives
modified distance functions $ h_{K,q} $ and $\tilde{h}_{K,\tilde{q}}
$.  For any $m \in U$, the signed distance $|qm|$ is a monotone
increasing function of $h_q(m)$, and distances from $\tilde{q}$ in
$Q_K$ have exactly the same relation with
$\tilde{h}_{K,\tilde{q}}$. Thus the following proposition shows that
curvature bounds imply triangle comparisons in the sense of
Proposition \ref{prop:equivcomp}.2, thereby proving the ``only if''
part of Theorem \ref{thm:comparison}.

\begin{prop} \label{prop:onlyif}
Set $h=h_{K,q} \circ\gamma_{pr}$ and $\tilde{h}
=\tilde{h}_{K,\tilde{q}} \circ\tilde{\gamma}_{\tilde{p}\tilde{r}}$.

If $R \ge K$ in $U$, then $h \ge \tilde{h} $.

If $R\le K$ in $U$, then $h \le \tilde{h}$.
\end{prop}
\begin{proof}
Assume $R \ge K$. Aside from reversing inequalities the proof for $R\le K$
is just the same.

Set $\gamma=\gamma_{pr}$ and
$\tilde{\gamma}=
\tilde{\gamma}_{\tilde{p}\tilde{r}}$.
For $m = \gamma(s)$, by
Corollary \ref{cor:Scompare2}, the modified shape operator
$S=S_{K,q}$ satisfies
\begin{equation*}
S(m) \le (1 - Kh_{K,q} (m))\cdot I.
\end{equation*}

Since, by definition,
    $\langle Sv,v\rangle$ is the second
derivative of $h_{K,q} $ along the geodesic with velocity $v$, then
\begin{equation*}
           (h_{K,q} \circ\gamma)''(s)
                     \le (1 - Kh_{K,q} (m)) \langle\gamma'(s), 
\gamma'(s)\rangle.
\end{equation*}
That is, along $\gamma$, $h_{K,q} $ satisfies the differential inequality
$$h'' + KE(\gamma)h \le E(\gamma).$$
On the other hand, the above inequalities become equations in $Q_K$, so
$$\tilde{h}'' + KE(\tilde{\gamma})\tilde{h}= E(\tilde{\gamma}).$$
But $E(\tilde{\gamma}) = E(\gamma)$ since $\tilde{\gamma}$ is a model
segment for $\gamma$. Hence the difference
$f = h - \tilde{h}$
is $KE(\gamma)$-concave:
$$f'' + KE(\gamma)f \le 0.$$
Moreover, at $0$ and $1$ the values of $h$ and $\tilde{h}$ are the same
since $E_q(p) = E_{\tilde{q}}(\tilde{p})$ and
$E_q(r) = E_{\tilde{q}}(\tilde{r})$, so the end values of $f$ are just
$f(0) = f(1) = 0$. By concavity $f$ is bounded below by the
$KE(\gamma)$-affine function with those end values, which is just
$0$.  That is,  $f \ge 0$, or
$h \ge \tilde{h}$.
\end{proof}

Next  we verify the ``if'' part of Theorem \ref{thm:comparison}:

\begin{prop}
If signed distances between pairs of points on any triangle in $U$
are at least
(at most)
those between
the corresponding points of the comparison triangle, then
$R \ge K$ ($R\le K$).
\end{prop}

\begin{proof}
Let $\sigma$ be a nonnull geodesic segment in $U$, let $v \in T_{\sigma(0)}M$
be nonnull and perpendicular to $\sigma'(0)$, and let $J$ be the Jacobi field
along $\sigma$ such that $J(0) = 0, J'(0) = v$. In the $2$-dimensional
model space $\tilde M$ of
curvature $K$ and of the same signature as the section spanned by
$\sigma'(0)$ and $v$,
choose a geodesic $\tilde\sigma$ and vector $\tilde v$ at $\tilde\sigma(0)$
perpendicular to $\sigma'(0)$ such that $\langle\tilde\sigma'(0),
\tilde\sigma'(0)\rangle = \langle\sigma'(0), \sigma'(0)\rangle$ and
$\langle \tilde v, \tilde v\rangle = \langle v, v \rangle$.  Let
$\tilde J$ be
the Jacobi field on $\tilde\sigma$ such that $\tilde J(0) = 0, \tilde
J'(0) = v$.

Write $$\tau(t,s)=\sigma_s(t)
=\exp_{\sigma(0)}t(\sigma'(0)+sv),$$
and similarly for $\tilde{\tau}$.  Since
$\frac{\partial\tau}{\partial t}(0,s)=\sigma'(0)+sv$, then
$\langle\frac{\partial\tau}{\partial t}(0,0),
\frac{\partial\tau}{\partial t}(0,s)\rangle$ is equal to the
corresponding expression in $\tilde{M}$.  But then our triangle
comparison assumption, in the form given in Proposition \ref
{prop:equivcomp}.3, and Lemma \ref{lem:hinge}.1 (Hinge) combine to give
         $|\sigma_0(t)\sigma_s(t)|
\le
|\tilde{\sigma}_0(t)\tilde{\sigma}_s(t)|.$
Since
$$|J(t)|=\lim_{s\to 0} |\sigma_0(t)\sigma_s(t)|/s,$$
and similarly in $\tilde{M}$, we conclude
$$ \langle J(t), J(t)\rangle \le \langle\tilde J(t), \tilde J(t)\rangle.$$

Now we calculate the third order Taylor expansion of $J$.
$$J'' = -R_{\sigma'J}\gamma', \quad J''(0) = 0,$$
$$J''' = -R'_{\gamma'J}\gamma' - R_{\gamma'J'}\gamma', \quad J'''(0)
= -R_{\gamma'(0)v}\gamma'(0),$$
and hence
$$J(t) = \mathcal{P}_t(vt - \frac{1}{6}R_{\gamma'(0)v}\gamma'(0)t^3 +
O(t^4)),$$
where $\mathcal{P}_t$ is parallel translation from $\gamma(0)$ to
$\gamma(t)$ and
the primes indicate $\nabla_{\gamma '(t)}$. Then we get an expansion
$$\langle J(t), J(t)\rangle
= \langle v, v\rangle t^2 -
\frac{1}{3}\langle R_{\gamma'(0)v}\gamma'(0), v\rangle t^4 + O(t^5),$$
and a similar expansion for $\langle\tilde J(t), \tilde J(t)\rangle$. Since the
$t^2$-terms are the same, we must have the inequality for the $t^4$-terms:
$$\langle R_{\gamma'(0)v}\gamma'(0), v\rangle
\ge \langle \tilde R_{\tilde\gamma'(0)\,\tilde v}\tilde\gamma'(0), v\rangle
= K\langle\gamma'(0), \gamma'(0)\rangle\langle v, v\rangle.$$
Since $\gamma'(0)$ and $v$ span an arbitrary nonnull section, $R \ge K$
follows.
\end{proof}


\section{Algebraic meaning of curvature bounds}
\label{sec:algebra}

Curvature bounds of the type studied in this paper are clarified by
the  analysis by Beem and Parker of the pointwise ranges of sectional
curvature \cite{BP}, as we now explain.  We go further, to relate our
curvature bounds to the ``null'' curvature bounds considered by
Uhlenbeck \cite{U} and Harris \cite{H1}.

Since in a semi-Riemannian manifold with indefinite metric, a
spacelike section always lies in a Lorentz or anti-Lorentz $3$-plane
$V$,  the range of sectional curvature may be studied by restricting
to such
$3$-planes $V$. On $V$, unless the curvature is constant, both the time-like
and space-like sections have
infinite intervals as their range, and either both are the entire real line
or both are rays which overlap in at most a common end (see
Theorem \ref{thm:Kvalues}). Then as we
vary $V$  in the tangent bundle, either the
separation between the two rays can be lost or we can have
numbers that separate all pairs of intervals, namely, a curvature
bound in our sense.

In this section, $V$ always denotes a
Lorentz or anti-Lorentz $3$-plane. Following \cite{BP}, consider a curvature
tensor $R$ on $V$.  Express $R$ as a homogeneous
quadratic form $v\wedge w \mapsto
\mathcal{Q}_1(v\wedge w) =R(v,w,v,w)$ on $\bigwedge^2V$.  If
$(e_1, e_2, e_3)$ is a frame for which $e_2$ and $e_3$ have the
same signature, then $(e_1\wedge e_2, e_1\wedge e_3, e_2\wedge e_3)$
is a frame for $\bigwedge^2V$ with signature $(-,-,+)$ with respect to
the natural extension of the inner product. Every nonzero element
$x_1e_1\wedge e_2+ x_2e_1\wedge e_3+ x_3e_2\wedge e_3$ of
$\bigwedge^2V$ is decomposable, and so represents a oriented section
of $V$, so the projective plane $\p^2$ of all nonorientable sections of
$V$ has homogeneous coordinates $x_1, x_2, x_3$. The inner product
quadratic form on $\bigwedge^2V$ has the coordinate expression
$\mathcal{Q}_2 = \langle v,v \rangle = (x_3)^2 - (x_1)^2 -(x_2)^2$, and the
sectional curvature function is $\mathcal{K} = \mathcal{Q}_1/\mathcal{Q}_2$. We
also identify $\mathcal{Q}_1$
and $\mathcal{Q}_2$ with the quadratic functions on $\p^2 - \{\ell_\infty\}$
given in terms of the corresponding nonhomogeneous coordinates
$x = x_1/x_3, \quad y = x_2/x_3$ by
$\mathcal{Q}_1/(x_3)^2$ and  $\mathcal{Q}_2/(x_3)^2 = 1 - x^2 - y^2$. For
various curvature
tensors there is no restriction on $\mathcal{Q}_1$;  that is,  for a
given point $p$ in
any $n$-dimensional manifold $M$, and a given $3$-dimensional subspace $V$
of $T_pM$,  a semi-Riemannian metric with indefinite restriction to $V$
can be specified in a neighborhood of $p$ in terms of normal coordinates
so as to realize any curvature tensor on $V$.

The {\em null conic} $N$ is given by $\mathcal{Q}_2 = 0$, and
represents those sections
of $V$ on which the inner product is degenerate and $\mathcal{K} =
\mathcal{Q}_1/\mathcal{Q}_2$ is undefined.
The {\em homaloidal (flat) conic} $H$ is given by $\mathcal{Q}_1 = 0$.
The inclusion $N
\subset H$ is equivalent to $\mathcal{K}$ being constant on the sections of
$V$, which is to say, $\mathcal{Q}_1$ being $\mathcal{Q}_2$
multiplied by that constant value (which may be $0$ so the inclusion
could be proper). Otherwise, $H$ and $N$ intersect in at most $4$
points, counting multiplicities.  The points of odd multiplicity are
precisely the points where $H$
and $N$ cross.

Since the interior and exterior of $N$ are connected sets on which
$\mathcal{K}$ is continuous, the ranges of $\mathcal{K}$ on
time-like sections and
space-like sections of $V$ are intervals, $I_{ti}$ and $I_{sp}$.
The following theorem characterizes the possible
ranges.  It implies, in particular, that if on $V$ either timelike
or spacelike curvatures are bounded, then both are, and there
exists a curvature bound in our sense.

\begin{thm}[\cite{BP}]\label{thm:Kvalues}
For a curvature tensor on a Lorentz or anti-Lorentz $3$-plane:
\begin{list}
{\emph{\arabic{step}.}}
{\usecounter{step}
\setlength{\rightmargin}{7mm}\setlength{\leftmargin}{7mm}}

\item $\mathcal{K}$ is constant if $N \subset H$.
\item $I_{sp} = I_{ti} = \R$ if $H$ and $N$ cross.
\item $I_{sp}$ and $I_{ti}$ are oppositely directed closed half-lines,
separated by a nontrivial open interval of curvature bounds, if $H$ does not
intersect $N$ (including the cases when $H$ is empty or a point not in $N$).
\item $I_{sp}$ and $I_{ti}$ are oppositely directed half-lines with a
common endpoint otherwise, namely, when $H$ and $N$ have a point of
tangency but never cross.   More specifically, $I_{sp}$ and $I_{ti}$
are both open, both closed, or complementary, according as $H$ and
$N$ intersect in a single point of order $2$, two points of order
$2$, or a single point of order $4$.
\end{list}
\end{thm}

In  a semi-Riemannian manifold with indefinite metric,
$R\ge K$ holds if and only if the restriction of the curvature tensor
to  each Lorentz or anti-Lorentz $3$-plane $V$ satisfies $R\ge K$
(and similarly for $R\le K$).
Equivalently, on each $V$, either $\mathcal{K}$ is constantly $K$, or $I_{ti}$
is a semi-infinite interval in
$(-\infty, K]$ and $I_{sp}$ is a semi-infinite interval in
$[K,\infty)$.  Theorem \ref{thm:Kvalues} leads us to
consider a weaker
condition, which we denote by $R_V\ge K(V)$, in which
the interval betweeen $I_{ti}$ and $I_{sp}$  varies with the
indefinite $3$-plane $V$, and
there may be no $K$ common to all.

Write $ R_{null} \ge 0 $ if
$R(v,x,v,x)\ge 0$ for any null vector $x$ and non-zero vector $v$
perpendicular to $x$.  It is shown in \cite[Proposition 2.3]{H1} (or
see \cite[Proposition A.7]{BEE})  that if $R_{null}>0\,
(<0)$ at a point, then the range of timelike sectional
curvatures at that point is unbounded below (above).  The following
proposition gives precise information.

\begin{prop}\label{prop:nullcurv}
A semi-Riemannian manifold with indefinite metric satisfies
$R_{null}\ge 0$ if and only if $R_V\ge K(V)$, and similarly with
signs reversed.
\end{prop}
\begin{proof}
In a given Lorentz or anti-Lorentz $3$-plane $V$, the condition
$R_{null} \ge 0$ is
equivalent to $\mathcal{Q}_1\ge 0$ on the null  conic $N$.
In turn this implies
that $N$ and $H$ do not cross, and hence cases 1, 3 or 4 of Theorem
\ref{thm:Kvalues} hold.  In case 1, obviously there is a lower
curvature bound. In cases 3 and 4, there are points of $N$ at which
$\mathcal{Q}_1>0$.  Approaching  $N$ from the spacelike side gives
$R\to\infty$, so  $I_{sp}$ is unbounded above and again $V$ has a
lower curvature bound.

Conversely, suppose there is a lower curvature bound for $V$, so case
2 is ruled out. In case 1, $\mathcal{Q}_1=0$ on $N$.  In cases 3 or 4, since
$I_{sp}$ is bounded below, there cannot be points of $N$ at which
$\mathcal{Q}_1< 0$. \end{proof}

The condition $R_{null}\le 0$ plus a ``growth
condition'' was used in \cite{U} to prove a Hadamard-Cartan theorem for
Lorentz manifolds.  It seems  interesting to investigate the
relation between $R\le 0$ and these hypotheses;  Uhlenbeck
comments about the growth condition,``it is to be hoped that a
similar condition that does not depend on coordinates can be found''
\cite[p. 75]{U}.

The condition $R_{null}> 0$
(or $< 0$) isolates case 3 of Theorem \ref{thm:Kvalues}.  Now let us show
how a strengthening of this condition bounds below
the length of the interval of curvature bounds in each
Lorentz or anti-Lorentz $3$-plane $V$.

While sectional curvature is
undefined for null sections, Harris has used a substitute, relative
to a choice of null vector $x$. Namely, for  a null section
$\Pi$ containing $x$, define  the \emph{null curvature of $\Pi$ with
respect to $x$} by
\begin{equation}
\mathcal{K}_x(\Pi) =R(w,x,w,x)/ \langle w,w\rangle
\end{equation}
for any non-null vector $w$ in $\Pi$  \cite{H1}.
While there is no \emph{a priori} way to normalize the null vector $x$,
it is still possible to strengthen Proposition
\ref{prop:nullcurv}.  This is because, in the presence of an interval
of curvature bounds larger than a single point, the algebra of the
curvature operator $\mathcal{R}:\bigwedge^2V\to\bigwedge^2V$ selects a
distinguished timelike unit vector $t$, or ``observer'', and hence a
distinguished circle of null vectors $x$.

In the following proposition, we suppose $V$ is Lorentz (that is,
has signature $(+,+,-)$).  There are obvious sign changes if
$-V$ is Lorentz.

\begin{prop} Suppose there is an interval $[K_1,K_2]$ of curvature bounds
below on the Lorentz
$3$-plane $V$, where $K_1<K_2$.  Then $\mathcal{R}$ is
diagonalizable.  Let $t$ be a
unit timelike vector perpendicular to the spacelike
eigenbivector of $\mathcal{R}$. Then
\begin{equation}
\label{eq:bdinterval}
K_2 - K_1 = \min_v K_x(\Pi),
\end{equation}
where $v$ runs over unit vectors perpendicular to $t$, and $x$ and
$\Pi$ are the null vector and null section $x=t+v$ and $\Pi=x^\perp$
respectively.  For curvature bounds above, substitute
\begin{equation}
K_1- K_2 = \max_v K_x(\Pi)
\end{equation}
for (\ref{eq:bdinterval}).
\end{prop}

\begin{proof}  We consider the case of curvature bounds
below.  First observe that, while self-adjoint linear operators in
indefinite inner product spaces are not always diagonalizable, our
hypotheses imply diagonalizability. Indeed, the unit eigenbivectors
of $\mathcal{R}$, of which one is spacelike and two are timelike, are
the critical points of the corresponding quadratic form on unit
bivectors.  The values of this quadratic form are sectional
curvatures, up to sign.  Therefore $K_2$, the minimum spacelike
sectional curvature, and $K_1$, the maximum timelike sectional
curvature, are eigenvalues, which are distinct by hypothesis.  The
corresponding eigenbivectors span a nondegenerate $2$-dimensional
subspace of $\bigwedge^2V$;  a bivector perpendicular to both is an
eigenbivector by self-adjointness.  Thus our eigenbivectors
diagonalize $\mathcal{R}$.  Let $t, v_1, v_2$ be a frame of vectors
perpendicular to the eigensections, so that $t\wedge v_1$ and
$t\wedge v_2$ are the timelike eigenbivectors.  Then the null vectors
$x=t+v$ have the form $t+\cos\theta v_1 + \sin\theta v_2$, and
the null curvatures $\mathcal{K}_x(\Pi)$  have the form  $K_2 -
K_1\cos^2\theta - K_3\sin^2\theta$ where $K_3\le K_1$.  Thus the
minimum is $K_2-K_1$.
\end{proof}

\section{Warped product examples}\label{sec:examples}

If $B$ and $F$ are Riemannian manifolds, $(-B)\times_fF$ will denote
the product manifold with the warped product metric $\langle
\,,\,\rangle = -ds^2_B +f^2
ds^2_F$. The sectional
curvature $\mathcal{K}$ of $(-B)\times_fF$, in terms of the sectional
curvatures $\mathcal{K}_B$ and $\mathcal{K}_F$, may be calculated for a
frame $x+v, y+w$, for $x,y \in T_pB$ and $v,w \in T_{\overline{p}}F$.
Without loss of generality, suppose $\langle x,y\rangle = \langle
v,w\rangle =0$.
Let $G$ be the gradient of $f$. Then
\begin{align*}
\mathcal{K}((x+v)&\wedge (y+w)) = -\mathcal{K}_B(x\wedge y) <x,x><y,y>\\
&- f^{-1}(p)\left[<w,w>
\nabla^2f(x,x)+ <v,v> \nabla^2f(y,y)\right]\\
&+f^{-2}(p)\left[\mathcal{K}_F(v\wedge w) - <G(p),G(p)>\right] <v,v> <w,w>.
\end{align*}
Therefore:

\begin{prop} \label{prop:wp} 
Consider Riemannian manifolds $B$ and $F$, and a smooth function
$f:B\to \R _{>0}$.  Then $(-B)\times _f F$ is a semi-Riemannian
manifold satisfying $R\ge K$ ($R\le K$) if and only if the
following three conditions hold:
\begin{list}
{\emph{\arabic{step}.}}
{\usecounter{step}
\setlength{\rightmargin}{7mm}\setlength{\leftmargin}{7mm}}

\item  $f$ is $(-K)$-concave\,\, ($(-K)$-convex).
\item  $\dim B = 1$ or $B$ has sectional curvature $\le -K$ ($\ge K$),
\item  $\dim F= 1$, or for all points $(p,\overline{p})$ and
$2$-planes $\Pi_{\overline{p}}$ tangent to $F$,

$\mathcal{K}_F(\Pi _{\overline{p}})\ge$ ($\le$) $ Kf(p)^2 + <G(p),G(p)>$.
\end{list}
\end{prop}

Taking $B$ to be an interval $I$ in Proposition \ref{prop:wp}, we
easily construct a rich class of Lorentz examples:

\begin{cor}\label{cor:wp}
If $f:I\to\R$ is $(-K)$-concave and $F$ is a Riemannian manifold of
sectional curvature $\ge C$, then $(-I)\times _f F$ satisfies $R\ge
K$ for any $K$ in the interval
\begin{equation}
\left[\sup\frac{f''}{f},\,\inf \frac{C+(f')^2}{f^2}\right].
\end{equation}

If $f:I\to\R$ is $(-K)$-convex and $F$ is a Riemannian manifold of
sectional curvature $\le C$, then $(-I)\times _f F$ satisfies $R\le
K$ for any $K$ in the interval
\begin{equation}
\left[\sup \frac{C+(f')^2}{f^2},\,\inf \frac{f''}{f}\right].
\end{equation}
\end{cor}

\begin{examp}
Following \cite{HE}, by a  \emph{Robertson-Walker space} we mean a
warped product $M=(-I)\times_f F$ where $F$ is $3$-dimensional
spherical, hyperbolic or Euclidean space, say with curvature $C$.
Then the sectional curvatures of sections containing
$\partial/\partial t$ are $K_-(t)=\frac{f''(t)}{f(t)}$, and those of
sections $\Pi$ tangent to the fiber are
$K_+(t)=\frac{C+f'(t)^2}{f(t)^2}$. By Corollary \ref{cor:wp},  $M$
satisfies $R\ge K$ if and only if $\sup K_- \le \inf K_+$.

It is easy to check that a Robertson-Walker space satisfies the
\emph{strong energy condition}, $\Ric(t,t)\ge 0$ for all timelike
vectors $t$, if and only if the curvature restricted to each tangent
$4$-plane has a nonpositive curvature bound below in our sense (see
\cite[Exercise 10, p. 362]{O'N}).

By  the Einstein equation, taking the cosmological constant
$\Lambda=0$, the stress-energy tensor of any Robertson-Walker space
has the form of a perfect fluid whose energy density $\rho$ and
pressure $p$ are functions of $t$ given by (see \cite[p. 346]{O'N}):
\begin{equation}\label{eq:RW}
8\pi\rho/3 = K_+, \hspace{.5in}
-4\pi(3p+\rho)/3=K_-.
\end{equation}

As discussed in  \cite[p. 348-350]{O'N}, the conditions $\rho >0$,
$\frac{-1}{3}<a\le \frac{p}{\rho}\le A$ for some constants $a$ and
$A$, and positive Hubble constant $H_0=\frac{f'}{f}(t_0)$ for some $t_0$,
correspond to an initial big bang singularity. Then $\rho <3a\rho\le
3p\le 3A\rho$, hence $0<(1+3a)\rho\le 3p+\rho$.  Therefore by
(\ref{eq:RW}), these big bang Robertson-Walker spaces all satisfy
$R\ge 0$.

Suppose the interval $I$ in these models is maximal.  If $C\le 0$,
then $I$ is semi-infinite and $\inf
\rho =0$, hence also $\inf p =0$,
so $0$ is the only curvature bound for the entire space. However,
every  point has a neighborhood which has an interval of
curvature bounds having $0$ as an interior point.
If $C>0$, then $f$ reaches a
maximum followed by a big crunch, and $K_+=\frac{C+(f')^2}{f^2}$
takes a positive minimum. Thus when $C>0$, the entire space has an
interval of curvature bounds with $0$ as an interior point.

Taking $\Lambda \ne 0$ here does not change the existence of
curvature bounds, but shifts them to the right by $\Lambda/3$.

In particular, a \emph{Friedmann model} is the special case in which
$\Lambda=0$ and
$p=0$.  Then one can solve explicitly for $f$, obtaining (see
\cite[p. 138]{HE}):
\begin{equation}
f=
\begin{cases}
\frac{ \mathcal{E} }{3} (\cosh\tau -1),\quad & t= \frac{\mathcal{E}
}{3} (\sinh\tau -
\tau), \text{ if } C=-1; \\
\tau ^2, \quad & t=\tau ^3/3, \text{ if } C=0;\\
-\frac{\mathcal{E} }{3} (1-\cos\tau), \quad & t =
-\frac{\mathcal{E} }{3} (\tau -
\sin\tau), \text{ if } C=1.
\end{cases}
\end{equation}
The  first two
of these solutions satisfy $R\ge 0$, and the third satisfies $R\ge K$
for all $K\in [-\frac{9}{8\mathcal{E} ^2}, \frac{9}{4\mathcal{E} ^2}]$.
\end{examp}

\begin{rem}
Vacuum spacetimes ($\Ric =0$) only have curvature bounds when they
are flat.  More generally, any $4$-dimensional Einstein Lorentz space
with a curvature bound has constant curvature, since perpendicular
sections always have the same curvature by a theorem of Thorpe
\cite{T}.
\end{rem}

\begin{examp}
We may also generate examples with higher index,
that is, higher-dimensional base. The following examples (a) and (b) of
curvature bounds for $(-B)\times_fF$ are from \cite{AH}:

(a) $R\ge K \,\,(\le K)$:  Take a Cartesian product $(-B)\times F$ (so $f=1$),
with sectional curvature$\le K$ in $B$ and $\ge K$ in $F$ \,\,(or the reverse).

(b) $R\ge 1 \,\,(\le 1)$: Take $B=H^k$, $f =
\cosh(\text{distance to a point})$, and $F$ of  sectional curvature
$\ge 1 \,\,(\le 1)$.

Note that to achieve $R\ge 1$ when $B$ is not $1$-dimensional, $B$
must have curvature $\le -1$.  Such a $B$ carries  many
$(-1)$-convex functions, but by Proposition \ref{prop:wp}, 
we need the warping function
$f$ on $B$ to be $(-1)$-concave.  A solution  is to take
$B=H^k$ and  $f$ to be  $(-1)$-affine. Example (b) fits this pattern, with
the righthand side of the inequality in Proposition \ref{prop:wp}.3
equal to $1$. Other  constructions in this pattern are:

(c) $R\ge 1 \,\,(\le  1)$:  Take $B=H^k$, $f=
\exp(\text{Busemann function})$, and $F$ of sectional curvature $\ge
0 \,\,(\le 0)$.

(d) $R\ge -1 \,\,(\le -1)$:  Take $B=S^k$, $f=\cos(\text{distance to
a point})$, and $F$ of  sectional curvature
$\ge -1 \,\,(\le -1)$.

Examples (a) - (d) are all geodesically complete.  Reversing  the
sign on an  example that satisfies $R\ge K$ and is negative definite
on the base, gives one that satisfies $R\le -K$ and is negative
definite on the fiber.
\end{examp}

\section*{Acknowledgments} We thank Yuri Burago for the picture that
triggered this project (\cite[p. 132]{BBI}).


\begin{thebibliography}{BFK}
\bibitem[AB1]{AB1}
S. B. Alexander, R. L. Bishop,
{\em $\mathcal{F}K$-convex functions on metric spaces}, Manuscripta
Math. {\bf 110} (2003), 115-133.

\bibitem[AB2]{AB2}
\bysame,
\emph{A cone splitting theorem for Alexandrov spaces}, Pacific Jour. Math.
{\bf 218} (2005), 1-16.

\bibitem[AH]{AH}
L. Andersson, R. Howard,
\emph{Comparison and rigidity theorems in semi-Riemannian geometry},
Comm. Anal. Geom. {\bf 6} (1998), 819-877.

\bibitem[AGH]{AGH}
L. Andersson, G. Galloway, R. Howard,
\emph{A strong maximum principle  for weak solutions of quasi-linear
elliptic equations with applications to Lorentzian and Riemannian
geometry}, Comm. Pure Appl. Math.{\bf 51} (1998), 819-877.

\bibitem [BP]{BP}
J. Beem, P. Parker,
{\em Values of pseudoriemannian sectional curvature},
Comment. Math. Helvetici {\bf 59} (1984), 319-331.

\bibitem [BE]{BE}
J. Beem, P. Ehrlich,
{\em Cut points, conjugate points and Lorentzian comparison theorems},
Math. Proc. Camb. Phil Soc. {\bf 86} (1979), 365-384.

\bibitem[BEE]{BEE}
J. Beem, P. Ehrlich, K. Easley,
{\em Global Lorentzian Geometry}, 2nd ed.
Dekker, New York, 1996.


\bibitem[BH]{BH}
M. Bridson, A. Haefliger,
{\em Metric Spaces of Non-positive Curvature},
Springer-Verlag, Berlin,1999.

\bibitem[BBI]{BBI}
D. Burago, Yu. Burago, S. Ivanov,
{\em A Course in Metric Geometry}, Graduate Studies in Mathematics,
Vol. 33, Amer. Math. Soc., Providence, 2001.

\bibitem[BGP]{BGP}
Yu. D. Burago, M. Gromov, G. Perelman,
{\em A. D. Alexandrov spaces with curvature bounded below},
Russian Math. Surveys {\bf 47} (1992), 1-58.

\bibitem[DGH]{DGH}
J. D\'{\i}az-Ramos, E. Garc\'{\i}a-R\'{\i}o, L. Hervella, {\em
Comparison results
for the volume of geodesic celestial spheres in Lorentzian
manifolds}, Diff. Geom. App. {\bf 23} (2005), 1-16.

\bibitem[EF]{EF}
J. Eells, B. Fuglede, {\em Harmonic Maps between Riemannian Polyhedra}.
With a preface by M. Gromov. Cambridge Tracts in Mathematics, 142.
Cambridge University Press, Cambridge, 2001.

\bibitem[F]{F}
F. Flaherty,
{\em Lorentzian manifolds of non-positive curvature},
Proc. Symp. Pure Math. {\bf XXVII}, part 2,  Amer. Math. Soc.,
Providence (1975), 395-399.

\bibitem[Gv]{Gv}
M. Gromov,
{\em Hyperbolic manifolds, groups and actions}.  I. Kra, B. Maskit
(eds.), {\em Riemann Surfaces and Related Topics}, Annals of Math.
Studies 97, Princeton University (1981), 183-213.


\bibitem[GvS]{GvS}
M. Gromov, R. Schoen,
{\em Harmonic maps into singular spaces and $p$-adic superrigidity for
lattices in groups of rank one}, Inst. Hautes Etudes Sci.
                Publ. Math. No. 76 (1992), 165--246.

\bibitem[Ge]{Ge}
K. Grove,
{\em Review of ``Metric Structures for Riemannian and non-Riemannian
spaces'' by M. Gromov}, Bull. Amer. Math. Soc, 38 (2001), 353-363.

\bibitem[H1]{H1}
S. Harris,
{\em A triangle comparison theorem for Lorentz manifolds}, Indiana
Math. J. {\bf 31} (1982), 289-308.

\bibitem[H2]{H2}
\bysame,
{\em On maximal geodesic diameter and causality in Lorentzian
manifolds}, Math. Ann. {\bf 261} (1982), 307-313.

\bibitem[HE]{HE}
S. Hawking, G. Ellis, {\em The Large Scale Structure of Space-Time},
Cambridge U. P., Cambridge,1993.

\bibitem[J]{J}
J. Jost, {\em Nonpositive Curvature: Geometric and Analytic Aspects},
Birkhauser, Basel, Boston, 1997.

\bibitem[Kr]{Kr}
H. Karcher, {\em Riemannian Comparison Constructions}. S. S. Chern, (ed.),
{\em Global Differential Geometry}, MAA Studies in Math., {\bf 27}, Math.
Assoc. Amer. 1987.

\bibitem[Ki]{Ki}
R. Kulkarni, {\em The values of sectional curvatures in indefinite
metrics}, Comment. Math. Helv. {\bf 54} (1979), 173-176.

\bibitem[O'N]{O'N}
B. O'Neill,
{\em Semi-Riemannian geometry with applications to relativity},
Academic Press, New York, 1983.

\bibitem[P]{P}
G. Perelman, {\em Alexandrov's spaces with curvature bounded from
below II}, preprint (1991).



\bibitem[T]{T}
J. Thorpe, {\em Curvature and the Petrov canonical forms}, J. Math. Phys.
{\bf 10} (1969), 1-7.

\bibitem[U]{U}
K. Uhlenbeck, \emph{A Morse theory for geodesics on a Lorentz
manifold}, Topology {\bf 14} (1975), 69-90.

\end{thebibliography}
\end{document}